\documentclass[a4paper]{amsart}

\newtheorem{theorem}{Theorem}[section]
\newtheorem{lemma}[theorem]{Lemma}
\newtheorem{problem}[theorem]{Open Problem}
\theoremstyle{definition}
\newtheorem{example}[theorem]{Example}

\numberwithin{equation}{section}

\newcommand{\dist}{{\mathop{\mbox{dist}}}}

\newcommand{\volk}{{\mathop{\lambda_k}}}
\newcommand{\voln}{{\mathop{\lambda_n}}}

\newcommand{\R}{\mathbb{R}}
\newcommand{\Sn}{S^{n-1}}
\newcommand{\al}{{\alpha}}
\newcommand{\be}{{\beta}}
\newcommand{\de}{{\delta}}
\newcommand{\eps}{{\varepsilon}}
\newcommand{\ga}{{\gamma}}
\newcommand{\la}{{\lambda}}
\newcommand{\De}{{\bigtriangleup}}
\newcommand{\calI}{{\mathcal I}}
\newcommand{\calS}{{\mathcal S}}
\newcommand{\calR}{{\mathcal R}}
\newcommand{\calT}{{\mathcal T}}

\begin{document}
\title[Convergence in shape]{Convergence in shape 
\\of Steiner symmetrizations}
\author[{G.~Bianchi, A.~Burchard, P.~Gronchi and A~Vol\v ci\v c}]{Gabriele Bianchi, Almut Burchard, Paolo Gronchi,\\ and Aljo\v{s}a Vol\v ci\v c}
\address{Dipartimento di Matematica, Universit\`a di Firenze, 
Viale Morgagni 67/A, Firenze, Italy I-50134}
\email{gabriele.bianchi@unifi.it}\email{paolo.gronchi@unifi.it}
\address{Department of Mathematics, University of Toronto, 40 St. George Street, Toronto, Ontario, Canada M5S 2E4}
\email{almut@math.toronto.edu}
\address{Dipartimento di Matematica, Universit\`a della Calabria, Ponte Bucci, cubo 30B, Arcavacata di Rende, Cosenza, Italy I-87036}
\email{volcic@unical.it}

\thanks{Supported in part by NSERC through Discovery Grant No. 311685-10.}

\begin{abstract} 
There are sequences of directions such that, given any compact set $K\subset\R^n$, the  sequence of iterated Steiner symmetrals of $K$ in these directions converges to a ball. 
However examples show that Steiner symmetrization along a sequence of directions whose differences are square summable does not generally converge.
(Note that this may happen even with sequences of directions which are dense in $\Sn$.) 
Here we show that such sequences converge \emph{in shape}.
The limit need not be an ellipsoid  or even a convex set.  

We also deal with uniformly distributed sequences of directions, and with a recent result of Klain on Steiner symmetrization along sequences chosen from a finite set of directions.
\end{abstract}
\maketitle

\section{Introduction}

Steiner symmetrization is often used to identify the ball
as the solution to geometric optimization problems.
Starting from any given body, one can find
sequences of iterated Steiner symmetrals
that converge to the centered ball of the same volume 
as the initial body. If the
objective functional improves along the sequence,
the ball must be optimal.

Most constructions of convergent sequences of
Steiner symmetrizations rely on auxiliary geometric
functionals that decrease monotonically along the 
sequence.  For example, the perimeter and the moment 
of inertia of a convex body decrease {\em strictly} under Steiner
symmetrization unless the body is already
reflection symmetric~\cite{Steiner,CS},
but for general compact sets, there are additional equality 
cases. The (essential) perimeter of a compact set 
decreases strictly under Steiner symmetrization in most, but not necessarily
all directions $u\in\Sn$, unless the set is a ball~\cite{CCF}.
Steiner symmetrization in an arbitrary direction strictly decreases
the moment of inertia, unless the set is already reflection 
symmetric up to a null set. 

Recently, several authors have studied how a sequence
of Steiner symmetrizations can fail to converge to the ball.
This may happen, even if the sequence
of directions is dense in $\Sn$.
Steiner symmetrizations of a convex body along any
dense sequence of directions can be made to converge
or diverge just by re-ordering~\cite{BKLYZ},
and
any given sequence of Steiner
symmetrizations (convergent or not) can be realized
as a subsequence of a non-convergent
sequence~\cite[Proposition 5.2]{BF}. 

In contrast, a sequence of Steiner symmetrizations
that uses only finitely many distinct directions always
converges~\cite{Klain}. 
The limit may be symmetric under
all rotations or under a non-trivial subgroup, depending
on the algebraic properties of those directions
that appear infinitely often.

A number of authors have studied Steiner
symmetrizations along random sequences of
directions. If the directions are chosen independently, uniformly
at random on the unit sphere, then the corresponding sequence of
Steiner symmetrals converges almost surely
to the ball simultaneously for all choices of the
initial set~\cite{Mani,vS,Volcic}.
Others have investigated the rate of convergence
of random and non-random sequences~\cite{BG,Klartag,Fortier,BF}.

We will address several questions that were
raised in these recent papers. The examples of
non-convergence presented there
use sequences of Steiner symmetrizations
along directions where the differences between successive angles are square summable.
Our main result, Theorem~\ref{main_theorem}, says that 
such sequences will converge if the Steiner symmetrizations
are followed by suitable rotations.  Convergence occurs both 
in Hausdorff distance and in symmetric difference.  The limit is
typically not an ellipsoid (or a convex set) unless the
sequence starts from an ellipsoid (or a convex
set, respectively). Some relevant examples and the statement of 
the theorem are contained in Section~\ref{section2}.

The proof of the theorem poses two technical challenges:
to show convergence of a sequence of symmetrals to
an unknown limit, rather than a ball;
and to show convergence in Hausdorff distance 
for an arbitrary compact initial set. 
This is more delicate than convergence 
in symmetric difference, because Steiner symmetrization 
is not continuous and Lebesgue measure is only upper 
semicontinuous on compact sets 
(see for instance \cite[p. 170]{Gru-book}). 

In Section~\ref{section3}, we collect the tools
to address these challenges.
Lemma~\ref{lem:symmdiff} relates
convergence of a sequence of compact sets 
in Hausdorff distance to convergence 
of their parallel sets in symmetric difference.
For sequences of Steiner symmetrals,
convergence in Hausdorff distance 
implies convergence in symmetric difference; 
in particular, the limit has the same measure as the initial set. 
To address the geometric problem of identifying
the limits of convergent subsequences, 
we use functionals of the form 
$$\calI_p(K)=\int_{\R^n} \phi(|x-p|)\, \psi(\dist(x,K))\, dx\,,$$
where $K$ is a compact set, $p$ a point in $\R^n$, the function
$\phi$ is increasing, and $\psi$ is decreasing. 
Then $\calI_p(K)$ decreases under simultaneous Steiner 
symmetrization of $K$ and $p$.
Note that by setting $p=o$,
$\phi(t)=t^2$ for all $t$, $\psi(0)=1$, 
and $\psi(t)=0$ for $t>0$, we recover the classical 
inequality for the moment of inertia.

We consider the special case where
$\phi$ and $\psi$ are the characteristic functions
of $[r,\infty)$ and $[0,\delta]$, respectively,
see~\eqref{mono}. 
Lemma~\ref{lem:one_three} implies that
for every pair of strictly monotone 
functions $\phi$ and $\psi$, the
functional decreases strictly
unless $K$ and $p$ agree with their Steiner symmetrals
up to a common translation.
By allowing $p\ne o$, we obtain information
about the intersection of the limiting
shape with a family of non-centered balls and half-spaces.
Lemma~\ref{lem:Radon} implies that
these intersections uniquely determine the shape.

In Section~\ref{section4}, we
combine the three lemmas to prove Theorem~\ref{main_theorem}.
It will be apparent from the proof that similar results 
should hold for other classical rearrangements. 
Lemmas~\ref{lem:symmdiff} and~\ref{lem:one_three}
remain valid for every rearrangement that
satisfies~\eqref{simm_and_setminus} and \eqref{simm_and_nbhd},
including the entire family
of cap symmetrizations studied by van Schaftingen~\cite{vS},
in particular polarization, spherical symmetrization,
and the Schwarz rounding process.

Lemmas~\ref{lem:symmdiff} and~\ref{lem:one_three}
are also useful for establishing convergence 
of Steiner symmetrals in Hausdorff distance 
in other situations, without the customary convexity 
assumption on the initial set.
In the remaining two sections,
we illustrate this with two more examples and 
pose some open questions. 

In Section~\ref{section5},
we consider Steiner symmetrization in the plane
along non-random sequences of directions that are
uniformly distributed (in the sense of Weyl) on $S^1$,
a property more restrictive than being dense.  
Theorem~\ref{teo:kronecker} shows that
a sequence of Steiner symmetrals along a Kronecker 
sequence of direction always converges  to a ball.
In the opposite direction, we give examples 
where convergence to a ball fails for certain uniformly 
distributed sequences. 

Finally, Section~\ref{section6} is dedicated to a recent result of Klain~\cite{Klain} on Steiner symmetrization along sequences chosen from a finite set of directions.
Klain proves that when $K$ is a convex body the sequence of Steiner symmetrals always converges. 
We extend this result to compact sets.

A.~Burchard wishes to thank F.~Maggi for inviting her to the University of Florence in the Spring of 2011, which led to this collaboration.

\section{Main Results}\label{section2}

We start with some definitions.  Let $o$ denote the origin in $\R^n$.
For $p\in\R^n$ and $r>0$, let $B_{r,p}$ denote the closed  ball of radius $r$ 
centered at $p$. If
$p=o$, we drop the second subscript and write simply $B_r$.
We write $\volk$ for the Lebesgue measure on $k$-dimensional 
subspaces of $\R^n$.
Directions in $\R^n$ are identified with unit vectors
$u\in \Sn$, and $u^{\perp}$ refers to the $(n-1)$-dimensional
subspace orthogonal to $u$. 

Let $K$ and $L$ be compact sets in $\R^n$. For $\de>0$, we denote by
$K_\delta=K+B_\delta$ the \emph{outer parallel set} of $K$.
The distance between compact sets will be measured
in the \emph{Hausdorff metric}, defined by
\[
d_H(K,L)=\inf \bigl\{\delta > 0 \mid
K\subset L_\delta \mbox{\,\,and\,\,} L \subset K_\delta\bigr\}\,.
\]
Another measure of the distance between $K$ and $L$
is their {\em symmetric difference distance} defined as
$\voln(K\De L)$. Note that this distance function
does not distinguish between sets that agree up to 
a null set. We will say that a sequence
of compact sets $(K_m)$ {\em converges 
in symmetric difference} to a compact set $L$, if
$$
\lim_{m\to\infty} \voln(K_m\De L)=0\,.
$$

Given a direction $u\in\Sn$, let $\calS_uK$ denote
the \emph{Steiner symmetral} of $K$ along $u$.  The
mapping $S_u$ that sends each set to its symmetral
is called \emph{Steiner symmetrization}.
We use here the variant that maps compact sets to compact sets,
which is defined as follows.
Denote by $\ell_y$  the line parallel to $u$ through
the point $y\in u^\perp$.  If $\ell_y\cap K$ is non-empty,
then $\ell_y\cap \calS_uK$ is the closed line segment of the same
one-dimensional measure centered at $y$;
if the measure is zero, the line segment degenerates to
a single point.  Otherwise,
$\ell_y$ intersects neither $K$ nor $\calS_u K$. 
Clearly, $\calS_u K$ is symmetric under reflection at $u^\perp$.

By Cavalieri's principle, the 
symmetral $\calS_uK$
has the same Lebesgue measure as the original set $K$. 
It is well known that Steiner symmetrization preserves convexity,
compactness, and connectedness, and that it
respects inclusions and 
reduces perimeter.  
For more information about Steiner symmetrization, we 
refer the reader to the book of Gruber~\cite[Chapter 9]{Gru-book}.

We have that $\calS_u (K\cap L)\subset \calS_u K\cap \calS_u L$,
since $\calS_u$ respects inclusions.
By writing $\voln(\calS_u K\setminus \calS_u L)$ as $\voln(\calS_u K)-\voln(\calS_uK\cap \calS_uL)$, this inclusion relation implies 
\begin{equation}\label{simm_and_setminus}
\voln(\calS_u K\setminus \calS_u L)\leq \voln(K\setminus L),
\end{equation}
an inequality that we repeatedly use.
It also implies that
$\voln (\calS_uK\De \calS_uL)\leq \voln (K\De L)$,
which means that
Steiner symmetrization is continuous in the symmetric 
difference metric
on the space of compact sets modulo null sets.

Moreover $\calS_u K+\calS_u L\subset \calS_u\left(K+L\right)$,
see \cite[Proposition 9.1(iii)]{Gru-book}. 
This in particular implies that for $\de>0$
\begin{equation}\label{simm_and_nbhd}
\left( \calS_u K \right)_\de\subset\calS_u K_\de.
\end{equation}

The following observation motivates our main result.

\begin{example}[Non-convergence]
\label{example_non_convergence}
Let $(\al_m)$  be a sequence  in $(0,\pi/2)$ with
\[
\sum_{m=1}^\infty\al_m=\infty\,,\qquad 
\sum_{m=1}^\infty\al_m^2<\infty\,,
\]
and set $\ga=\prod_{m=1}^\infty\cos\al_m$. Note that $\ga\in(0,1)$.
For each positive integer $m$, let $\be_m=\sum_{k=1}^m\al_k$
and $u_m=(\cos\be_m,\sin\be_m)$. 

Let $K$ be a convex body that has area smaller than a disc of 
diameter $\gamma$ and contains a vertical line 
segment $\ell$ of length~1.  Apply the sequence of 
Steiner symmetrizations $S_{u_m}$ to $K$ and $\ell$
to obtain a sequence of convex bodies $K_m$ and line segments $\ell_m$.
Each symmetrization $\calS_{u_m}$
projects the previous line segment $\ell_{m-1}$
onto $u_m^\perp$, thereby multiplying its
length by $\cos \al_m$. Since $\be_m$ diverges,
the segments $\ell_m$ spin in circles forever while 
their length decreases monotonically to $\gamma$.

For each $m$, the diameter of $K_m$
exceeds $\gamma$, because $K_m\supset \ell_m$.
If the sequence converges, its
limit must contain a disc of diameter $\gamma$.
On the other hand its area equals that of $K$, a contradiction.
\hfill$\Box$
\end{example}

It turns out that the sequences from
Example~\ref{example_non_convergence}
\emph{converge in shape}, in the sense that there exist 
isometries $\calI_m$ such that 
$(\calI_m \calS_{u_m}\dots \calS_{u_1}K)$ converges for each compact set $K$. The sequence $(\calI_m)$ depends only on $(u_m)$.

\begin{theorem}\label{main_theorem}
Let  $(u_m)$ be a sequence  in $\Sn$
with $u_{m-1}\cdot u_m=\cos\al_m$, 
where $(\al_m)$  is a sequence  in $(0,\pi/2)$
that satisfies
$\sum_{m=1}^\infty\al_m^2<\infty$.
There exists a sequence of rotations $(\calR_m)$
such that for every non-empty compact set $K\subset\R^n$,
the rotated symmetrals
\begin{equation}\label{rotated_sequence}
 K_m=\calR_m\calS_{u_m}\dots \calS_{u_1}K
\end{equation}
converge in Hausdorff distance to a compact set~$L$.  Moreover,
$$
\lim_{m\to\infty} \voln(K_m\De L)=0\,.
$$
\end{theorem}

\medskip
What can we say about the limit  
of the sequence $(K_m)$ in \eqref{rotated_sequence}?
Since Steiner
symmetrization transforms ellipsoids into ellipsoids, 
one may wonder whether the 
limit is always an ellipsoid~\cite{BKLYZ,Klain}.
The following examples show that this is not the case.

\begin{example}[The limit need not be an ellipse]\label{limit_non_ellipse}
Let the sequence $(u_m)$ and $\ga=\prod_{m=1}^\infty\cos\al_m$ be as in 
Example~\ref{example_non_convergence}. Observe 
that dropping, if necessary, a few initial terms we can make 
$\sum_{m=1}^\infty\al_m^2$ arbitrarily small and 
hence $\gamma$ arbitrarily close to 1.
In particular we may suppose  $\gamma>2/\pi$.  
Let $K$ be the convex envelope of the line segment $\ell$ from
Example \ref{example_non_convergence} 
and a centered ball $B_r$ for some $r>0$. If
the sequence  \eqref{rotated_sequence} converges, 
its limit contains both $B_r$ and a line segment 
of length $\gamma$.  Any ellipse that contains these sets has area 
at least $\pi \ga r/2$.  On the other hand  its area agrees 
with the area of $K$, which is bounded from above by $r/\sqrt{1-4r^2}$,
the area of the rhombus circumscribed to the circle centered at $o$ 
whose longer diagonal is a
segment of length $1$. Since this is smaller than
$\pi \ga r/2$ if $r$ is small enough, by our choice of $\gamma$, 
the limit cannot be an ellipse.
\hfill $\Box$
\end{example}

\begin{example}[The limit can be non-convex]
\label{limit_non_convex}
Take the sequence $(u_m)$ as in Example~\ref{example_non_convergence},
and let $K$ be the union of a line segment $\ell$ and a
ball $B_r$.  The limit of the sequence \eqref{rotated_sequence} 
contains $B_r$ and a line segment of length $\gamma$. 
Any convex set that contains these sets has area at least 
$\ga r/2$.  Since the area of $K$ is $\pi r^2$, 
the limit cannot be a convex set if $\pi r<\ga/2$.
\hfill $\Box$
\end{example}

\section{Some lemmas} \label{section3}

Our first lemma relates convergence in the Hausdorff 
metric to convergence in symmetric difference.

\begin{lemma}
\label{lem:symmdiff} Let $L$ and $K_m$, $m\geq1$,  
be  non-empty compact sets. 
\begin{enumerate}
 \item\label{lem:symmdiff_i} The sequence $(K_m)$ converges in 
Hausdorff distance to $L$ if and only if 
$$
 \lim_{m\to\infty} \voln((K_m)_\de\De L_\de)=0\
$$
for each $\de>0$.

\item\label{lem:symmdiff_ii} If $(K_m)$ converges
in Hausdorff distance to $L$ and 
each $K_m$ is obtained  from a compact set $K$ via 
finitely many Steiner symmetrizations and 
Euclidean isometries, then
$$
 \lim_{m\to\infty} \voln(K_m\De L)=0\,.
$$
In particular, $\voln(L)=\voln(K)$.
\end{enumerate}
\end{lemma}

\begin{proof} For Claim (i), assume that $K_m$ 
converges to $L$ in Hausdorff distance.
Fix $\de>0$, and let $\eps>0$ be given.
Since $L$ is compact, $\voln(L_\rho)$ is continuous in 
$\rho>0$.  Choose $\rho\in (0,\de)$ so small that 
$\voln(L_{\de+\rho})-\voln(L_\de)<\eps$ 
and $\voln(L_\de)-\voln(L_{\de-\rho})<\eps$,
and let $m$ be so large that $d_H(K_m,L)<\rho$.
Then $d_H((K_m)_\de,L_\de)<\rho$ for each $\de>0$.
It follows that
$$
\voln\left((K_m)_\de\setminus L_\de \right) \leq 
 \voln(L_{\de+\rho}\setminus L_\de ) <\eps
$$
since $(K_m)_\de\subset L_{\de+\rho}$, and
$$
\voln\left(L_\de\setminus(K_m)_\de \right)\leq  
\voln\left(L_{\de}\setminus L_{\de-\rho} \right)<\eps
$$
since $L_{\de-\rho}\subset (K_m)_\de$.
Combining the two inequalities, we obtain that
$$\voln((K_m)_\de\De L_\de))<\eps$$
for $m$ sufficiently large. Since 
$\eps>0$ was arbitrary, convergence in symmetric difference follows.

To see the converse implication, assume that
$d_H(K_m,L)\geq 2\rho>0$. If
$K_m\setminus L_{2\rho}\neq\emptyset$ 
then $(K_m)_\rho\setminus L_\rho$ 
contains $B_{\rho,p}$, where $p$ is any point in 
$K_m\setminus L_{2\rho}$. Otherwise, $L_\rho\setminus (K_m)_\rho$ 
contains $B_{\rho,p}$, where $p$ is any point in 
$L\setminus (K_m)_{2\rho}$.
In either case, $\voln((K_m)_\rho\De L_\rho)\geq \voln(B_\rho)$.

For Claim (ii),
assume that $(K_m)$ converges to $L$ in
Hausdorff distance. Given $\eps>0$, choose $\rho>0$ so small 
that $\voln(L_{\rho})-\voln(L)<\eps$, and choose $m$
so large that $d_H(K_m,L)<\rho$.
Then $K_m\subset L_\rho$, and therefore
$$
\voln\left(K_m\setminus L \right) \leq 
 \voln(L_{\rho}\setminus L ) <\eps\,.
$$
For the complementary inequality, construct $(K_\rho)_m$ by 
applying the same sequence of symmetrizations and isometries 
to the parallel set $K_\rho$ that was used to produce $K_m$. Then, 
by \eqref{simm_and_nbhd} and since symmetrization does 
not change volume,  we have
\begin{align*}\label{haus_simmdif3}
 \voln\left(L\setminus K_m \right)&\leq  
 \voln\left((K_m)_{\rho}\setminus K_m \right)\\
&\leq\voln\left((K_\rho)_{m}\right)-\voln\left(K_m \right) \\
&=\voln\left(K_{\rho}\right)-\voln\left(K \right)\\
&<\eps\,.
\end{align*}
Combining the two preceding inequalities, we
conclude as in the proof of the first claim
that $\voln(K_m\De L)$ converges to zero.
\end{proof}

\medskip
The second part of Lemma~\ref{lem:symmdiff} could have been
proved by using \eqref{simm_and_nbhd} and a  result of 
Beer~\cite[Theorem 1]{Be2}, 
who also, in~\cite[Lemma 4]{Be1}, proved the ``only if'' 
implication of the first part of the lemma.

The next lemma provides an equality statement
for~\eqref{simm_and_setminus} in the case where one
set runs through the family of parallel sets
$K_\de$ and the other set is a ball $B_{r,p}$.

\begin{lemma} 
\label{lem:one_three}
Let $u\in\Sn$, and let $K$ be a non-empty compact set. 
If there exists a point $p\in u^\perp$ such that
\begin{equation}\label{mono-2}
\voln( \calS_u K_\de \setminus B_{r,p})
= \voln(K_\de\setminus B_{r,p})
\end{equation}
for all $\de,r>0$, then $\calS_uK=K$.
\end{lemma}

\begin{proof}
Suppose $\calS_uK\ne K$, and fix
a point $q\in K\setminus \calS_uK$. Let $y\in u^\perp$
be such that the line  $\ell_y$ parallel to
$u$ and passing through $y$ contains $q$.  By definition,
$\calS_uK$ intersects $\ell_y$ in a centered line segment of 
the same one-dimensional measure as $K\cap\ell_y$.
Since $p\in u^\perp$, we can choose $r<|p-q|$ such that
$B_{r,p}$ contains $S_uK\cap \ell_y$
in its interior.

We argue that the boundary of $B_{r,p}$ separates
a neighborhood of $q$ from $S_uK\cap \ell_z$ for $z$ close
to $y$.  For $\de>0$, consider the nested compact sets 
\[
A(\de)=S_u K_\de\cap (\ell_y)_\de.
\]
By the compactness of $K$, the one-dimensional measure 
of each cross section $K_\de\cap \ell_z$ converges 
monotonically to the measure of 
$K\cap \ell_z$ as $\de$ tends to zero,
and hence $\bigcap_{\de>0}S_u K_\de =S_uK$.
It follows that
$$
\bigcap_{\de>0} A(\de) = 
\left( \bigcap_{\de>0} S_uK_\de\right) \cap 
\left(\bigcap_{\de>0} (\ell_y)_\de\right)  = \calS_uK\cap \ell_y\,.
$$
Since  $\calS_uK\cap \ell_y$ does not meet 
$\{x: |x-p|\ge r\}$, by compactness
there exists a set $A(\de)$ that does not meet 
$\{x:|x-p|\ge r\}$ either. This means that 
the interior of $B_{r,p}$ contains $\calS_uK_\de \cap (\ell_y)_\de$
for some $\de>0$.
By choosing $\de>0$ small enough, we can further assume that
$B_{r,p}$ does not intersect $B_{\de,q}$.

By construction,
\begin{equation}\label{part1}
 \voln\bigl(\bigl(\calS_u K_\de\cap (\ell_y)_\de\bigr) 
\setminus \calS_u B_{r,p}\bigr)
<\voln\bigl(\bigl(K_\de\cap (\ell_y)_\de\bigr)\setminus B_{r,p}\bigr),
\end{equation}
because the set on the left-hand side 
is empty, while the one on the right-hand side contains $B_{\de,q}$. 
Let $C$ be the closure of $\R^n\setminus (\ell_y)_\de$.
Since
$\calS_u K_\de\cap C =\calS_u (K_\de\cap C)$,
it follows from \eqref{simm_and_setminus} that
\begin{equation}\label{part2}
\voln((\calS_u K_\de\cap C) \setminus \calS_u B_{r,p})
\leq\voln((K_\de\cap C) \setminus B_{r,p})\,.
\end{equation}
Adding \eqref{part1} and \eqref{part2} and using
that $(\ell_y)_\de$ and $C$ form an almost disjoint
partition of $\R^n$, we obtain that
$$
\voln( \calS_u K_\de \setminus \calS_u B_{r,p})
< \voln(K_\de\setminus B_{r,p})\,,
$$
negating \eqref{mono-2}.
\end{proof}

For the proof of the main result, we will
combine \eqref{simm_and_setminus} with~\eqref{simm_and_nbhd} 
to obtain
\begin{equation}\label{mono}
\voln( (\calS_u K)_\de \setminus \calS_u B_{r,p})
\le \voln(K_\de\setminus B_{r,p})
\end{equation}
for each $p\in \R^n$ and all $\de,r>0$.
Lemma~\ref{lem:one_three} implies that
for every given $p\in u^\perp$, the inequality
in \eqref{mono} is strict for some $\delta,r>0$ unless $\calS_uK=K$.

The last lemma will be used to identify the
limit of \eqref{rotated_sequence}.

\begin{lemma} 
\label{lem:Radon}
Let $H_1,H_2$ be compact sets in $\R^n$, and let $u\in\Sn$. 
Assume that $\calS_{u}H_j=H_j$ for $j=1,2$, and that
\begin{equation} \label{semispazi}
\voln(H_1\cap\{x\cdot p>t\}) = 
\voln(H_2\cap\{x\cdot p>t\})
\end{equation}
for all non-zero $p\in u^\perp$ and all $t\in\R$.
Then $H_1$ and $H_2$ agree up to a set of $n$-dimensional
Lebesgue measure zero.
\end{lemma}

\begin{proof} Denote by $\ell_y$ the line parallel to $u$
through $y\in u^\perp$, and consider on $u^\perp$ the
measurable functions $f_j(y)=\la_1(\ell_y\cap H_j)$
for $j=1,2$. 
By assumption, the difference $f_1-f_2$ integrates to zero
over every half-space $\{y\cdot p>t\}\subset u^\perp$. 
It follows from a standard argument
that its integral over
almost every $(n-2)$-dimensional subspace
$\{y\cdot p=t\}\subset u^\perp$ vanishes as well~\cite{Z}.
In other words, the $(n-2)$-dimensional Radon transform
of $f_1-f_2$ is zero almost everywhere, and 
therefore $f_1=f_2$ almost everywhere on $u^\perp$~\cite[p.28]{H}.
Since $\calS_uH_j=H_j$, the sets
are uniquely determined by the functions $f_j$,
and we conclude that $H_1$ and $H_2$ agree up to
a null set.
\end{proof}

\section{Proof of Theorem~\ref{main_theorem}}\label{section4}

We begin with some geometric
considerations.  Given $u\in\Sn$, we
want to compose a Steiner symmetrization $\calS_u$  
with a rotation $\calR'$ so that the result is symmetric at the 
hyperplane $e_1^\perp$. Note that the
commutation rule  
\begin{equation}\label{comm_rule}
 \calR\calS_u = \calS_{\calR u}\calR\,
\end{equation}
holds for every rotation $\calR \in O(n)$ and every $u\in\Sn$.  

Let $u\in \Sn$, with $u\cdot e_1=\cos\alpha$.
Replacing $u$ with $-u$, if necessary, we may take $\alpha\in (0,\pi/2)$. 
The Steiner symmetrization $\calS_u$ projects 
subsets of $e_1^\perp$ linearly onto $u^\perp$.
If $\calR'$ is the rotation
that maps $u$ to $e_1$ and fixes $u^\perp \cap e_1^\perp$,
then the composition $\calT=\calR'\calS_u$ defines
a linear transformation on $e_1^\perp$ that satisfies
\begin{equation} \label{T-bound}
|\calT x - x|\le (1-\cos\alpha)|x|\,,\quad
|x|\ge |\calT x|\ge |x| \cos\alpha 
\end{equation}
for all $x\in e_1^\perp$.
More precisely, the restriction of $\calT$ to $e_1^\perp$
is equivalent to a diagonal matrix with eigenvalues 
$\cos\alpha$ (simple) and $1$ (of multiplicity $n-2$).

Given a non-empty compact set $K$, let $(S_{u_m})$ be 
a sequence of Steiner symmetrizations as described 
in the statement of the theorem,
and let $(K_m)$ be the sequence of rotated symmetrals
defined in the statement of the theorem.
We construct the rotation $\calR_m$ in \eqref{rotated_sequence}
as a composition $\calR_m=\calR'_m\dots \calR_1'$,
where $\calR'_m$ is recursively defined as the rotation
that sends $\calR_{m-1}u_m$ to $e_1$ 
and fixes the subspace orthogonal to these two vectors, and $\calR_0'=\calI$. Note that $\calR'_m$ is a rotation by $\alpha_m$.
By the commutation rule \eqref{comm_rule},
\[
\calR_{m+1} \calS_{u_{m+1}} 
= \calR_{m+1}'\calS_{\calR_m u_{m+1}}\calR_m\,,
\]
which gives for $(K_m)$ the recursion relation
$$
K_{m+1}= \calR_{m+1}'\calS_{\calR_m u_{m+1}} K_m\,.
$$

By Blaschke's selection principle, $(K_m)$ has subsequences 
that converge in Hausdorff distance.  
Let $L_1$ and $L_2$ be limits of such subsequences.  
We want to prove that $L_1=L_2$.

We will first show that
\begin{equation}\label{cerchi}
\voln((L_1)_\de\setminus B_{r,q}) = 
\voln((L_2)_\de\setminus B_{r,q})
\end{equation}
for all $q\in e_1^\perp$ and all $r,\de\ge 0$. 
By the assumption that
$(\al_m)$ is square summable, $\gamma=\prod_{m=1}^\infty \cos \alpha_m>0$.
Let $\calT_m'$ be the linear transformation 
defined by $\calR_m'\calS_{\calR_{m-1} u_m}$ on $e_1^\perp$
and consider the composition $\calT_m=\calT_m'\dots \calT_1'$. 
By~\eqref{T-bound},
the sequence $(\calT_m)$ converges to a linear transformation
${\calT}$ that satisfies
$|\calT x|\ge \gamma |x|$ for all $x\in e_1^\perp$.
In particular, $\calT$ is invertible on $e_1^\perp$.
For each $m\geq 1$,  let $p_m=\calT_m\calT^{-1}q$. Then
$p_m=\calT_m' p_{m-1}$, $p_m$ converges to $q$ and 
$$
 \calR'_m\calS_{\calR_{m-1}u_m}B_{r,p_{m-1}}
=B_{r, p_{m}}.
$$
Inequalities~\eqref{simm_and_setminus} and~\eqref{simm_and_nbhd} 
imply that the sequence
$\voln\bigl((K_m)_\de\setminus B_{r,p_m}\bigr)$
is monotonically decreasing, 
\begin{equation*}\begin{aligned}
\voln((K_m)_\de\setminus B_{r,p_m})
&\ge \voln(\calS_{\calR_mu_{m+1}}(K_m)_\de\setminus 
\calS_{\calR_mu_{m+1}}B_{r,p_m})\\
&\ge \voln((\calS_{\calR_mu_{m+1}}K_m)_\de\setminus 
\calS_{\calR_mu_{m+1}}B_{r,p_m})\\
&= \voln((K_{m+1})_\de\setminus B_{r,p_{m+1}}),
\end{aligned}\end{equation*}
hence convergent.
In the last line, we have used the rotational
invariance of Lebesgue measure and the recursion formula
for $K_m$.  Passing to the limit 
along the subsequences converging to $L_1$ and $L_2$
and using Lemma~\ref{lem:symmdiff} yields~\eqref{cerchi}. 

Since half-spaces can be written as increasing unions of balls,
we can take a monotone limit in \eqref{cerchi}
to obtain that \eqref{semispazi} holds, with $H_i=(L_i)_\de$, for all
$q\in e_1^\perp$ and all $t\in\R$.
Lemma~\ref{lem:Radon} implies that 
the parallel sets $(L_1)_\de$ and $(L_2)_\de$ 
agree up to a null set. 
To complete the proof,
suppose that $L_1\ne L_2$. Then there exists a point
$x$ that lies in one of the two sets but not the other,
say $x\in L_1\setminus L_2$. If we choose
$\de = \frac{1}{2} \dist(x,L_2)$, then
$(L_1)_\de\supset B_{\de,x}$ 
while $(L_2)_\de\cap B_{\de,x}=\emptyset$. This is
impossible since the parallel sets agree up to a null set,
and we conclude that $L_1=L_2$.  
\hfill$\Box$

\section{Uniformly distributed sequences}\label{section5}

A sequence $(u_m)$ in $S^1$ is called
\emph{uniformly distributed} in the sense of Weyl, 
if the fraction of terms in 
the initial segment $(u_m)_{m\le N}$ that fall into
any given arc $I$ in $S^1$ converges to ${\lambda_1(I)}/({2\pi})$
as $N$ tends to infinity, where $\lambda_1(I)$ is the length of $I$. 
A classical example is the Kronecker sequence $u_m=(\cos m\alpha,
\sin m\alpha)$ for $m\ge 1$, which is uniformly distributed
if $\alpha$ is not a rational multiple of $\pi$.

\begin{theorem} [The Kronecker sequence] \label{teo:kronecker}
Let $u_m=(\cos m\alpha,\sin m\alpha)$ for $m\ge 1$,
and assume that $\alpha$ is not a rational 
multiple of $\pi$. Let $K$ be a non-empty compact set.
Then the symmetrals $\calS_{u_m}\dots\calS_{u_1}K$
converge in Hausdorff distance and in symmetric difference
to the closed centered ball $K^*$ equimeasurable with~$K$.
\end{theorem} 

\begin{proof} Let $\calR$ be the 
rotation that sends $u=(\cos\alpha, \sin\alpha)$ 
to $e_1$, and let $\calS=\calS_{e_1}$ be the Steiner 
symmetrization in the direction of $e_1$. 
It suffices to show that
\[
K_m  = \calR^m\calS_{u_m}\dots \calS_{u_1}K
\]
converges to $K^*$. 

By the commutation relation \eqref{comm_rule}, $K_m=(\calS\calR)^mK$. 
Let $(K_{j_m})$ be a subsequence that converges in Hausdorff
distance to a compact set $L$. 
By Lemma~\ref{lem:symmdiff}, the sequence $\bigl((K_{j_m})_\de\bigr)$ 
converges in symmetric difference to $L_\de$ 
for each $\de> 0$.  We estimate
\begin{equation}
\label{eq:mono}\begin{aligned}
\voln(L_\de \setminus B_r) &=\voln(\calR L_\de \setminus  B_r)\\
&\ge\voln(\calS\calR L_\de \setminus \calS B_r)\\
&=\lim_{m\to\infty} \voln(\calS\calR (K_{j_m})_\de\setminus B_r)\\
&\ge \inf_m \voln((K_{j_m+1})_\de\setminus B_r)\,.
\end{aligned}\end{equation}
The inequality in the second line follows 
from \eqref{simm_and_setminus}.
The third line uses the continuity of the Steiner symmetrization 
with respect to the symmetric difference metric and that
$\calS B_r=B_r$.  The fourth line  follows from
\eqref{simm_and_nbhd} and the identity
$\calS\calR K_{j_m}=K_{j_m+1}$. 

On the other hand, the sequence $\voln((K_m)_\de\setminus B_r)$
is monotone decreasing, for each $\de,r>0$, by 
\eqref{simm_and_nbhd} and \eqref{simm_and_setminus},
hence convergent along the entire sequence,
and by Lemma~\ref{lem:symmdiff}, its limit is given by 
$\voln\left( L_\de\setminus B_r\right)$.
This means that all inequalities in
\eqref{eq:mono} hold with equality and, in particular, 
\[
 \voln(\calR L_\de \setminus B_r)= \voln(\calS\calR L_\de \setminus \calS B_r).
\]

By construction,
$L$ is symmetric under reflection at $e_1^\perp$.
By Lemma~\ref{lem:one_three} 
and the fact that $\calR L_\de=(\calR L)_\de$, we have 
$\calS\calR L=\calR L$, i.e., $L$ is also symmetric
under reflection at $u^\perp$.  Since
$\alpha$ is incommensurable with $\pi$, these two 
reflections generate a dense subgroup of rotations, and we 
conclude that $L=K^*$.  Since the subsequence was arbitrary,
the entire sequence $(K_m)$ converges to $K^*$ in Hausdorff 
distance and in symmetric difference.
\end{proof}

One may wonder if every uniformly distributed
sequence of directions gives rise to a convergent sequence of
Steiner symmetrizations.  Since a sequence of directions chosen 
independently and uniformly at random from $S^1$
is almost surely uniformly distributed 
\cite[Theorem 3.2.2]{KN} and
the corresponding Steiner symmetrizations 
almost surely converge to the ball~\cite{Mani,vS,Volcic}, 
most uniformly distributed sequences of directions
produce convergent sequence of Steiner symmetrizations.

Remarkably, there are exceptions. In the notation 
of Example~\ref{example_non_convergence},  let
$(\alpha_m)$ be a nonincreasing sequence of 
positive numbers 
and set $\beta_m=\sum_{k=1}^m \al_k$ for $m\ge 1$.
If
\[
\lim_{m\to\infty} \al_m=0\,,\qquad
\lim_{m\to\infty} m\al_m=\infty\,, \qquad
\text{and}\quad
\sum_{m=1}^\infty\alpha_m^2<\infty\,, \]
then 
$(u_m)=(\cos\be_m,\sin\be_m)$ is uniformly 
distributed on $S^1$ (see~\cite[Theorem 2.5]{KN}). 
This includes in particular
sequences of the form $\al_m=\vartheta m^{-\sigma}$ 
with $\sigma\in (1/2,1)$ and $\vartheta>0$.
But the corresponding sequence of Steiner symmetrals
of the compact set in Example~\ref{example_non_convergence}
does not converge.

Uniformly distributed sequences play an
important role in quasi-Monte Carlo methods,
because they share many properties of random 
sequences.  In some cases, they
provide even better approximations to 
integrals than typical random sequences.
The quality of the approximation defined by
a sequence $(u_m)$ in $S^1$ is determined by its
{\em discrepancy}
$$
D(N)=\sup_{I\subset S^1} 
\left\vert
\frac{\# \{m\le N: u_m\in I\}}{N} - \frac{\la_1(I)}{2\pi}\right\vert\,,
$$
which describes how much the fraction of
the initial segment $(u_m)_{m\le N}$ that fall into
any given arc $I$ in $S^1$ differs from
${\lambda_1(I)}/({2\pi})$. The best approximations are provided by
sequences of {\em minimal discrepancy}, that is by sequences for which $D(N)$ is proportional to
$(\log N)/N$. Note that the discrepancy of a
Kronecker sequence depends on the diophantine 
properties of $\alpha/\pi$, and that the discrepancy of the sequence $(\cos \beta_n,\sin \beta_n)$ which corresponds to  $\al_n=\vartheta n^{-\sigma}$, with $\sigma\in (1/2,1)$ and $\vartheta>0$, has asymptotic behavior  $n^{-\sigma}$ \cite{KN}.

\begin{problem} 
If $(u_m)$ is uniformly distributed in $S^1$ and of minimal discrepancy, do the Steiner symmetrals $\calS_{u_m}\dots \calS_{u_1}K$ converge to $K^*$ for each compact set $K$?
\end{problem}

\section{Steiner symmetrization along a finite
set of directions}\label{section6}

\medskip Our final example
concerns sequences of iterated Steiner 
symmetrization that use finitely many directions. 
Klain proved the elegant result that  when the 
initial set is a convex body then such sequences  
always converge~\cite[Theorem 5.1]{Klain}. 
The techniques developed in this paper allow 
us to extend his result to compact sets. 

\begin{theorem} [Klain's Theorem holds for
compact sets]\label{teo:klain}
Let $(u_m)$ be a sequence of vectors chosen from a finite set
$F=\{v_1,\dots, v_k\}\subset\Sn$. Then, for every compact set
$K\subset\R^n$, the symmetrals 
\[
K_m= \calS_{u_m}\dots \calS_{u_1}K
\]
converge in Hausdorff distance and in symmetric difference
to a compact set~$L$.  Furthermore, $L$ 
is symmetric under reflection in each of the directions 
$v\in F$ that appear in the sequence infinitely often.
\end{theorem}

\begin{proof}We follow Klain's argument. Dropping an initial segment 
$(K_m)_{m\le N}$ of the sequence and possibly substituting $F$ with one of its subsets,
we may  assume, without loss of generality,  that
each direction in $F$ appears infinitely often
in the sequence $(u_m)$.
The main idea is to construct a 
subsequence along which the directions $v_i\in F$
appear in a particular order.  
With each index $m$, we associate a permutation 
$\pi_m$ of the numbers $1,\dots, k$
that indicates the order in which the directions
$v_1,\dots, v_k$ appear for the first time among the directions $u_i$ with $i\geq m$.  Since there are only finitely many permutations,
we can pick a subsequence $(u_{j_m})$ such 
that the permutation $\pi_{j_m}$ is the same for each $m$.
By re-labeling the directions,
we may assume that this permutation is the identity.
Passing to a further subsequence, we may assume that
every direction in $F$ appears in each segment 
$u_{j_m}, u_{j_m+1},\dots, u_{j_{m+1}}$.

By the Blaschke selection principle,
there is a subsequence (again denoted by $(K_{j_m})$)
that converges in Hausdorff distance to some 
compact set $L$.  We note for later use that
for each $\de>0$, the entire sequence $\bigl(\voln((K_m)_\de)\bigr)$ is 
decreasing by \eqref{simm_and_nbhd}, hence convergent.
By Lemma~\ref{lem:symmdiff}, the limit is given by
\begin{equation}\label{mono-0}
\inf_m \voln((K_m)_\de)=\lim_{m\to\infty}
\voln((K_{j_m})_\de) =\voln(L_\de)\,.
\end{equation}

We show by induction that $S_{v_i}L=L$
for $i=1,\dots, k$.  For $i=1$ observe that $u_{j_m}=v_1$. Therefore $(K_{j_m})$ is symmetric with respect to $v_1^\perp$ and the same is true for $L$.
Suppose we already know that $L$ is
invariant under Steiner symmetrization in
the directions $v_1,\dots v_{i-1}$. If
$j'_m$ is the index where $v_i$ appears for the first time after $j_m$,
then the inductive hypothesis implies that
$\calS_{u_{j'_m-1}}\dots \calS_{u_{j_m+1}} L=L$. 
By~\eqref{simm_and_setminus} and~\eqref{simm_and_nbhd} we have, 
for each $\de>0$,
\begin{equation}\label{eq:Klain-1}
\begin{aligned}
\voln((K_{j_m})_\de\setminus L_\de)
&\ge \voln(\calS_{u_{j'_m-1}}\dots \calS_{u_{j_m+1}}(K_{j_m})_\de 
\setminus L_\de)\\
&\ge\voln((\calS_{u_{j'_m-1}}\dots
\calS_{u_{j_m+1}}K_{j_m})_\de\setminus L_\de)\\ 
&= \voln((K_{j'_m-1})_\de\setminus L_\de)\,.
\end{aligned}
\end{equation}
Since $((K_{j_m})_\de)$ converges to $L_\de$ in symmetric difference,
the left hand side of the inequality converges to zero.
Therefore, the right hand side converges to zero,
and by \eqref{mono-0}, the sequence
$(K_{j'_m-1})_\de$ likewise converges to $L_\de$ in symmetric difference.
We continue the argument as in \eqref{eq:mono} and estimate
\begin{equation}\label{eq:mono2}
\begin{aligned}
\voln(L_\de \setminus B_r) &\ge \voln(\calS_{v_i}L_\de \setminus \calS_{v_i}B_r)\\
&= \lim_{m\to\infty} \voln(\calS_{v_i}(K_{j'_m-1})_\de\setminus B_r)\\
&\ge\inf_m \voln((K_{j'_m})_\de\setminus B_r)\,.
\end{aligned}
\end{equation}
The inequality in the first line follows from
\eqref{simm_and_setminus}. 
In the second line we have used
the convergence of $(K_{j_m'-1})_\de$ in symmetric difference
and the continuity of the Steiner symmetrization
with respect to the symmetric difference distance.
The inequality in the third line is a consequence of 
\eqref{simm_and_nbhd} and of the equality $S_{v_i} K_{j'_m-1}=K_{j'_m}$.

Since $\bigl(\voln((K_m)_\de\setminus B_r)\bigr)$ is a decreasing sequence
by \eqref{simm_and_setminus}, and since it contains the  
subsequence $\bigl(\voln((K_{j_m})_\de\setminus B_r)\bigr)$ 
which converges to  
$\voln(L_\de\setminus B_r)$,  the first and last term 
in~\eqref{eq:mono2} are equal.  In particular,
the first line holds with equality for each $\de>0$.
By Lemma~\ref{lem:one_three}
this implies $L=\calS_{v_i}L$, which concludes the inductive step.

It remains to prove that the entire sequence converges. 
Since $L$ is invariant under Steiner symmetrization
in each of the directions $v_1,\dots , v_k$ in $F$, we have, 
by the same reasoning as in \eqref{eq:Klain-1} and in the 
lines following it,  that 
\[
\voln((K_{j_m})_\de\setminus L_\de)
\ge \voln((K_j)_\de\setminus L_\de)
\]
for every $j\ge j_m$. We conclude that 
$(K_m)_\de$ converges to $L_\de$ in 
symmetric difference along the entire sequence, for each $\de\ge 0$. 
By Lemma~\ref{lem:symmdiff}, 
$(K_m)$ converges to $L$ both in Hausdorff distance and 
in symmetric difference.
\end{proof}

\begin{problem} 
Do iterated Steiner symmetrals $\calS_{u_m}\dots \calS_{u_1}K$
always converge in shape,
without any assumptions on the sequence of directions? 
\end{problem}

\begin{problem}
Assume that a sequence of directions $(u_m)$ is such that  $(\calS_{u_m}\dots \calS_{u_1}C)$ converges to $C^*$ for each convex body $C$. 
Is it true that  $(\calS_{u_m}\dots \calS_{u_1}K)$ converges to $K^*$ for each compact set $K$?
\end{problem}





\begin{thebibliography}{BK}
\bibliographystyle{plain}

\bibitem{Be1}
G.~A.~Beer, {Starshaped sets and the Hausdorff metric}, Pacific J. Math. {\bf 69} (1955), 21--27.

\bibitem{Be2}
G.~A.~Beer, {The Hausdorff metric and convergence in measure}, Michigan  Math. J. {\bf 21} (1974), 63--64.

\bibitem{BG}
G.~Bianchi and P.~Gronchi, {Steiner symmetrals and their distance from a ball}, Israel J. Math. {\bf 135} (2003), 181--192.

\bibitem{BKLYZ}
G.~Bianchi, D.~Klain, E.~Lutwak, D.~Yang, and G.~Zhang, {A countable set of directions is sufficient for Steiner symmetrization}, Adv. in Appl. Math. {\bf 47} (2011), 869--873. 

\bibitem{BF}
A.~Burchard and M.~Fortier, {Random polarizations}, arXiv:1104.4103v3 (2011).

\bibitem{CS}
C.~Carath{\'e}odory and E.~Study, {Zwei {B}eweise des {S}atzes, da\ss\ der {K}reis unter allen {F}iguren gleichen {U}mfanges den gr\"o\ss ten {I}nhalt hat}, Math. Ann. {\bf 68} (1909), 133--140.

\bibitem{CCF} 
M.~Chleb\'{i}k, A.~Cianchi, and N.~Fusco, {The perimeter inequality under {S}teiner symmetrization: Cases of equality},  Ann. of Math. (2) {\bf 162} (2005), 525--555.

\bibitem{Fortier}
M.~Fortier, {Convergence results for rearrangements: Old and new}, M.S. Thesis, University of Toronto, December 2010.


\bibitem{Gru-book}
P.~Gruber, {{C}onvex and {D}iscrete {G}eometry},  Springer Verlag, New York, 2007.

\bibitem{H} 
S.~Helgason, {The Radon Transform}, Birkh\"auser, Boston, 1999.


\bibitem{Klain}
D.~Klain, {Steiner symmetrization using a finite set of directions}, Adv. in Appl. Math.  {\bf 48} (2012), 340--353.

\bibitem{Klartag}
B.~Klartag, {Rate of convergence of geometric symmetrizations}, Geom. Funct. Anal. {\bf 14} (2004), 1322--1338.

\bibitem{KN} L. Kuipers and H. Niederreiter, {Uniform Distribution of Sequences}, Wiley-Interscience, New York, 1974.

\bibitem{Mani}
P.~Mani-Levitska, {Random {S}teiner symmetrizations}, Studia Sci. Math Hungar. {\bf 21} (1986), 373--378.

\bibitem{Steiner}
J.~Steiner, {Einfacher {B}eweise der isoperimetrischen {Haupts\"atze}}, J. Reine Angew. Math. {\bf 18} (1838), 281--296.
 
\bibitem{vS}
J.~{van Schaftingen}, {Approximation of symmetrizations and symmetry of critical points}, Topol. Methods Nonlinear Anal. {\bf 28} (2006), 61--85.

\bibitem{Volcic}
A.~Vol\v{c}i\v{c}, Random Steiner symmetrizations of sets and functions, Calc. Var. Partial Differential Equations, to appear, doi:10.1007/s00526-012-0493-4.

\bibitem{Z} L. Zalcman, {Offbeat integral geometry}, Amer. Math. Monthly {\bf 87} (1980), 161--175.

\end{thebibliography}
\end{document}